\documentclass{article}
\usepackage{amssymb}
\usepackage{amsmath}
\usepackage{amsthm}
\usepackage{mathabx}
\usepackage{enumerate}
\usepackage{xcolor}
\usepackage{changepage}
\usepackage{graphicx}

\usepackage{hyperref}

\theoremstyle{plain}
\everymath=\expandafter{\the\everymath\displaystyle}

\newtheorem{Lem}{Lemma}
\newtheorem{Cor}{Corollary}

\newtheorem{Thm}{Theorem}

\newcommand*{\Ker}{\ensuremath{\mathrm{Ker\,}}}
\newcommand*{\Ann}{\ensuremath{\mathrm{Ann\,}}}
\newcommand*{\Hom}{\ensuremath{\mathrm{Hom\,}}}
\newcommand*{\R}{\ensuremath{\mathbb{R}}}

\newcommand*{\Z}{\ensuremath{\mathbb{Z}}}
\newcommand*{\C}{\ensuremath{\mathbb{C}}}

\begin{document}

	\date{}

\author{
	L\'aszl\'o Sz\'ekelyhidi\\
	{\small\it Institute of Mathematics, University of Debrecen,}\\
	{\small\rm e-mail: \tt szekely@science.unideb.hu,}
}	
	\title{Spectral Analysis Implies Spectral Synthesis}

	\maketitle
	
	\begin{abstract}
	In this paper we show that spectral analysis implies spectral synthesis for arbitrary varieties on locally compact Abelian groups which have no discrete subgroups of infinite torsion free rank.
	\end{abstract}
	
	\footnotetext[1]{The research was supported by the  the
		Hungarian National Foundation for Scientific Research (OTKA),
		Grant No.\ K-134191.}\footnotetext[3]{AMS (2000) Subject Classification: 43A45, 22D99}
	\bigskip\bigskip


\section{Introduction}

Let $\mathcal A$ be a commutative topological algebra and $X$ a topological vector module over $\mathcal A$. This means that $X$ is a topological vector space and also a module over $\mathcal A$ so that the action $(a,x) \mapsto  a\cdot x$ is continuous in both variables. We say that {\it spectral analysis} holds for $X$, if every nonzero closed submodule of $X$ contains a finite dimensional submodule. By simple linear algebra, this is equivalent to the property that every nonzero closed submodule contains a common eigenvector for $\mathcal A$. On the other hand, we say that {\it spectral synthesis} holds for $X$, if every closed submodule of $X$ is the closure of the sum of all its finite dimensional submodules.  Clearly, spectral synthesis implies spectral analysis. 
\vskip.2cm

This terminology comes from the classical results of L.~Schwartz about spectral synthesis (see \cite{MR0023948}). In that case $\mathcal A$ is the {\it measure algebra} $\mathcal M_c(\R)$: the space of all compactly supported complex Borel measures on $\R$ with the linear operations and with the convolution product. The topology of $\mathcal M_c(\R)$ is the weak*-topology, when $\mathcal M_c(\R)$ is considered the topological dual of the space $\mathcal C(\R)$, the space of all continous complex valued functions on $\R$ equipped with the topology of uniform concergence on compact sets. Then $\mathcal C(\R)$ is a topological vector module over $\mathcal M_c(\R)$, where the latter acts on the function space by the convolution of compactly supported measures with continuous functions. Here the closed submodules are exactly the so-called {\it varieties}: these are the closed, translation invariant linear subspaces of $\mathcal C(\R)$. The common eigenvectors of $\mathcal M_c(\R)$ are exactly the exponential functions, and the finite dimensional varieties are exactly those consisting of {\it exponential polynomials}, which are finite linear combinations of exponential functions with polynomial coefficients. Schwartz's famous theorem reads as follows:

\begin{Thm} (L.~Schwartz, 1947) 
	Spectral synthesis holds for $\mathcal C(\R)$.
\end{Thm}

The above setting is general enough to cover several interesting cases. In \cite{MR0098951}, M.~Lefranc proved spectral synthesis on $\Z^n$ for each natural number $n$, where obviously $\mathcal C(\Z^n)$ is the set of all complex valued functions on $\Z^n$. A natural question is: what about spectral synthesis for $\mathcal C(G)$, where $G$ is an arbitrary locally compact Abelian group? In \cite{MR2340978}, the authors characterized those discrete Abelian groups, where spectral synthesis holds. On the other hand, it turned out that Schwartz's result cannot be generalized to several real variables: in \cite{MR0390759}, D.~Gurevich presented a variety in $\R^2$, for which spectral analysis does not hold, and another one, which contains finite dimensional subvarieties, but those do not span a dense subvariety.  
\vskip.2cm

A complete description of those locally compact Abelian groups where spectral synthesis holds for the space of all continuous functions was obtained in \cite{Sze23c}, where the present author proved the following two theorems:

\begin{Thm}
	Spectral synthesis holds on the compactly generated locally compact Abelian group if and only if it is topologically isomorphic to $\R^a\times Z^b\times C$, where $C$ is compact, and $a,b$ are nonnegative integers with $a\leq 1$.
\end{Thm}

\begin{Thm}
	Spectral synthesis holds on the locally compact Abelian group $G$ if and only if $G/B$ is topologically isomorphic to $\R^a\times Z^b\times F$, where $B$ is the subgroup of all compact elements in $G$, $F$ is a discrete Abelian group of finite rank, and $a,b$ are nonnegative integers with $a\leq 1$.
\end{Thm}

In the last two theorems the condition $a$ is $0$ or $1$ is the consequence of Gurevich's counterexamples. In both examples the given variety contains finite dimensional varieties, still they do not span the whole variety. The point is, what we have to underline, that the existence of finite dimensional subvarieties in a variety does not mean that spectral analysis holds for the variety: the requirement for spectral analysis is that all nonzero subvarieties must contain finite dimensional subvarieties. In Gurevich's examples clearly, even spectral analysis fails to hold for the given varieties: not {\it every} nonzero subvariety includes finite dimensional subvarieties.
\vskip.2cm

This observation leads to the following problem: does there exist a variety on a locally compact Abelian group, for which spectral analysis holds, but spectral synthesis fails to hold? In the discrete case, the result in  \cite{MR2340978} says that spectral synthesis holds on {\bf every variety} on the group if and only if the torsion-free rank of the group is finite, but in the earlier paper \cite{MR2120269} it has been proved that on a discrete Abelian group spectral analysis holds on {\bf every variety} if and only if the torsion-free rank is less than the continuum. Here the key is: every variety. But what about a single variety? What is the exact connection between spectral analysis and spectral synthesis for a single variety? We know that synthesis implies analysis -- but what about the converse? This question has not been answered yet, not even by the characterization theorems, nor by the counterexamples. The purpose of this paper is to present a somewhat unexpected result: in fact, spectral analysis and spectral synthesis are equivalent for any varieties on locally compact Abelian groups, if the group does not contain a discrete subgroup of infinite torsion free rank.

\section{Preliminaries}
Here we summarize some known results we shall use in the subsequent paragraphs. 
\vskip.2cm

Given a locally compact Abelian group $G$ the continuous complex homomorphisms of $G$ into the multiplicative group of nonzero complex numbers, resp. into the additive group of complex numbers are called {\it exponentials}, resp. {\it additive functions}. A product of additive functions is called a {\it monomial}, and a linear combination of monomials is called a {\it polynomial}. A product of an exponential and a polynomial is called an {\it exponential monomial}, and if the exponential is $m$, then we call the exponential monomial an {\it $m$-exponential monomial}. Hence polynomials are exactly the $1$-exponential monomials. Linear combinations of exponential monomials are called {\it exponential polynomials}.  One dimensional varieties are exactly those spanned by an exponential, and finite dimensional varieties are exactly those spanned by exponential monomials (see \cite{MR3185617}). The {\it variety of the function $f$} in $\mathcal C(G)$, denoted by $\tau(f)$, is the intersection of all varieties including $f$.
\vskip.2cm

Annihilators play an important role in spectral analysis and synthesis. For each closed ideal $I$ in $\mathcal M_c(G)$ and for every variety $V$ in $\mathcal C(G)$, $\Ann I$ is a variety in $\mathcal C(G)$ and $\Ann V$ is a closed ideal in  $\mathcal M_c(G)$. Further, we have
$$
\Ann \Ann I=I\enskip\text{and}\enskip \Ann \Ann V=V.
$$
(see \cite[Section 11.6]{MR3185617}, \cite[Section 1]{MR3502634}).

\begin{Thm}\label{egyes}
	Let $G$ be a locally compact Abelian group. The following statements are equivalent.
	\begin{enumerate}[1.]
		\item $M$ is a closed maximal ideal in $\mathcal M_c(G)$.
		\item The residue ring of $M$ is the complex field.
		\item $\Ann M$ is a one dimensional variety in $\mathcal C(G)$.
		\item $M$ is the annihilator of the variety of an (unique) exponential.
		\item  $M$ is the closure of the ideal generated by all measures $\delta_{-y}-m(y)\delta_0$ for $y$ in $G$, where $m$ is a (unique) exponential.
	\end{enumerate}
\end{Thm}
(See \cite[Section 2]{MR3502634}, \cite{MR3185617}). The annihilator of the variety of the exponential $m$ is denoted by $M_m$.

\begin{Cor}\label{kettes}
	Let $G$ be a locally compact Abelian group and $V$ a variety on $G$. The  following statements are equivalent.
	\begin{enumerate}[1.]
		\item Spectral analysis holds for the variety $V$.
		\item Every maximal ideal in the residue ring of $\Ann V$ is closed.
	\end{enumerate}
\end{Cor}

\begin{Thm}\label{harmas}
	Let $G$ be a locally compact Abelian group and $M_m$ a closed maximal ideal in $\mathcal M_c(G)$. The following statements are equivalent.
	\begin{enumerate}[1.]
		\item The function $f$ in $\mathcal C(G)$ is an $m$-exponential monomial.
		\item The variety $\tau(f)$ is finite dimensional, and there is anatural number $k$ such that $M_m^k$ is in $\Ann \tau(f)$.
	\end{enumerate}
\end{Thm}
(See \cite[Theorem 12, Theorem 13]{MR3502634}). 

\begin{Cor}\label{negyes}
	Let $G$ be a locally compact Abelian group and $V$ a variety on $G$. The  following statements are equivalent.
	\begin{enumerate}[1.]
		\item Spectral synthesis holds for the variety $V$.
		\item The annihilator $\Ann V$ is the intersection of all closed ideals $I$ satisfying $\Ann V\subseteq I$ whose residue ring is a local Artin ring.
		\item We have
		$$
		\Ann V=\bigcap_{\Ann V\subseteq M_m} \bigcap_{k=0}^{\infty} M_m^{k+1},
		$$
		where the first intersection is extended to all exponentials $m$ in $V$ for which $\mathcal M_c(G)\big/M_m^{k+1}$ is finite dimensional, for each $k$.
	\end{enumerate}
\end{Cor}(See e.g. \cite[Theorem 16]{MR3713567}).
\vskip.2cm
Spectral synthesis is also related to another, more general concept of polynomials: the so-called {\it generalized polynomials}. Given an exponential maximal ideal $M_m$ in the measure algebra $\mathcal M_c(G)$ of the locally compact Abelian group $G$, the functions in $\Ann M_m^{k+1}$ are called {\it generalized exponential monomials of degree at most $k$}, for each natural number $k$, which are called {\it generalized polynomials}, if $m=1$. For $k=0$ these are the constant multiples of $m$. If $k=1$, then every generalized exponential monomials of degree at most $k$ has the form $x\mapsto (a(x)+b)m(x)$, where $a$ is additive and $b$ is a complex number -- in fact, this is an $m$-exponential monomial. However, for $k\geq 2$ there may exist generalized exponential monomials of degree at most $k$, which are not exponential monomials, as their variety is infinite dimensional (see \cite[Sections 12.5, 12.6]{MR3185617}). For the variety of a generalized exponential monomial of degree at least 2 spectral analysis clearly holds, but this variety is not synthesizable (see \cite[Corollary 15.2.1]{MR3185617}). The existence of generalized exponential monomials which are not exponential monomials dependes on the torsion free rank of a discrete Abelian group. In fact, on a discrete Abelian group, every generalized exponential polynomal is an exponential polynomial if and only if the torsion free rank of the group is finite (see \cite[Corollary 13.2.2]{MR3185617}). It follows that, when studying the relation between spectral analysis and spectral synthesis on varietie, it is reasonable to assume that the underlying locally compact Abelian group does not contain a discrete subgroup of infinite torsion free rank. This is equivalent to the assumption that $\Hom(G,\C)$ is a finite dimensional vector space. In other words, there are only finite many linearly independent additive functions on $G$.

\begin{Thm}\label{tor}
	Let $G$ be a locally compact Abelian group which does not contain a discrete subgorup of infinite torsion free rank. Then every generalized exponential monomial is an exponential monomial on $G$. Conversely, if every generalized exponential monomial is an exponential monomial on $G$, then $G$ does not contain a discrete subgorup of infinite torsion free rank.
\end{Thm}

\begin{proof}
	Assume that $G$ is a locally compact Abelian group which does not contain a discrete subgorup of infinite torsion free rank. Let $B$ denote the closed subgroup of all compact elements in $G$. Clearly, every additive function is zero on $B$, consequently every generalized polynomial is constant on $B$. It follows that every generalized exponential monomial on $B$ is a constant multiple of an exponential.  Let $\Phi:G\to G/B$ denote the natural homomorphism. It follows that $\varphi$ is a generalized exponential monomial on $G/B$ if and only if $\varphi\circ \Phi$ is a generalized exponential monomial on $G$. On the other hand, $G/B$ is topologically isomorphic to $\R^n\times F$, where $n$ is a natural number and $F$ is a discrete torsion free Abelian group (see \cite[(24.35) Corollary]{MR0156915}). By our assumption, $F$ is of finite rank, hence it is a homomorphic image of $\Z^k$ for some natural number $k$, which implies that there are only infinitely many linearly independent additive functions on $\R^n\times F$. It follows that $G$ has the same property.
	\vskip.2cm
	
	For the converse, we refer to \cite[Corollary 13.2.2]{MR3185617}).
\end{proof}

\section{Spectral analysis implies spectral synthesis}
Given a locally compact Abelian group $G$ a {\it representation of the measure algebra $\mathcal M_c(G)$} on the locally convex topological vector space $X$ is a continuous algebra homomorphism of $\mathcal M_c(G)$ onto an algebra of continuous linear operators of $X$. Here $X$ is called the {\it the representation space}. In order to avoid trivial cases we always assume that the measure $\delta_0$ is mapped onto the identity operator. Every representation of $\mathcal M_c(G)$ makes a natural vector module structure on $X$ with the action $x\mapsto \pi(\mu)\cdot x$, where $\pi$ is the representation and $\pi(\mu)$ is the linear operator on $X$ corresponding to $\mu$. By the {\it canonical representation} of $\mathcal M_c(X)$ we mean the one with representation space $\mathcal C(X)$ and $\mu$ is mapped onto the linear operator $f\mapsto \mu*f$ for each $\mu$ in $\mathcal M_c(G)$ and $f$ in $\mathcal C(G)$. Given a variety $V$ in $\mathcal C(G)$ the {\it canonical representation  of $\mathcal M_c(G)$ on the variety $V$} is the one where the representation space is $V$, and the linear operators $f\mapsto \mu*f$ are restricted to $V$. 

\begin{Lem}
Let $\pi$ be a representation of $\mathcal M_c(G)$ on the locally convex topological vector space $X$. Then $\Ker \pi$ is a closed  ideal in $\mathcal M_c(G)$. Conversely, every closed closed ideal in $\mathcal M_c(G)$ is the kernel of a representation of $\mathcal M_c(G)$.
\end{Lem}

\begin{proof}
Clearly, $\Ker\pi$ is a closed linear space in $\mathcal M_c(G)$. As $\pi$ is an algebra homomorphism, it follows that $\Ker\pi$ is closed under multiplication. Further, for each $\mu$ in $\Ker\pi$ and $\nu$ in $\mathcal M_c(G)$, we have
$$
\pi(\mu*\nu)(x)=\pi(\mu)(x)\cdot \pi(\nu)(x)=0,
$$
which shows that $\Ker\pi$ is a closed ideal.
Conversely, let $I$ be a closed ideal in $\mathcal M_c(G)$, then $\mathcal A=\mathcal M_c(G)\big/I$ is a locally convex topological vector space which is also a topological algebra. Clearly, the natural mapping $\Phi$ of $\mathcal M_c(G)$ onto $\mathcal A$ is an algebra homomorphism. On the other hand, if $V$ denotes the annihilator of $I$ in $\mathcal C(G)$, then $V$ is a variety, and $\mathcal A$ is topologically isomorphic to the measure algebra of $V$. As $\mathcal M_c(G)\big/I$ is topologically isomorphic to $\mathcal A$, for any topological isomorphism $\iota$ from $\mathcal M_c(G)/I$ to $\mathcal A$   the mapping $\pi=\iota\circ \Phi$ is a representation of $\mathcal M_c(G)$ on $V$ with \hbox{$\Ker\pi=I$.}
\end{proof}

Let $\pi_1, \pi_2$ be two representations of $\mathcal M_c(G)$. We say that {\it $\pi_2$ is a subrepresentation of $\pi_1$}, if $\Ker \pi_1\subseteq \Ker \pi_2$. For instance, if $V_1,V_2$ are varieties on $G$, then $\pi_{V_1}$ is a subrepresentation of $\pi_{V_2}$ if and only if $V_1\subseteq V_2$.

\begin{Lem}
	Let $V$ be a variety on $G$. Then spectral analysis holds for $V$ if and only if every maximal ideal in $\mathcal M_c(G)$, which  includes the kernel of some nonzero subrepresentation of the canonical representation of $\mathcal M_c(G)$ on $V$, is closed.
\end{Lem}

\begin{proof}
	By Corollary \ref{kettes}, spectral analysis holds for a variety $V$ if and only if every maximal ideal, which includes $\Ann V$, is closed.  Suppose that spectral analysis holds for $V$, and $\pi$ is a subrepresentation of $\pi_V$, the canonical representation of $\mathcal M_c(G)$ on $V$. The measure $\mu$ is in $\Ker\pi_V$ if and only if $\mu*f=0$ for each $f$ in $V$, hence $\Ker \pi_V$ is the annihilator of $V$. Assume that $M$ is a maximal ideal in $\mathcal M_c(G)$ with $M\supseteq \Ker \pi$. Then we have $M\supseteq \Ker \pi_V=\Ann V$, consequently $M$ is closed.
	\vskip.2cm
	
	The converse is obvious, as $\pi_V$ is a subrepresentation of itself, hence every maximal ideal including $\Ker \pi_V=\Ann V$ is closed, which is equivalent to spectral analysis for $V$.
\end{proof}

Let $V$ be a variety on $G$ and $W\subseteq V$ a subvariety. The natural representation of $\mathcal M_c(G)$ on $V/W$ is defined as follows: for each $f$ in $V$ we let
$$
\pi_{V/W}(\mu)\cdot (f+W)=\mu*f+W.
$$
Clearly, $\pi_{V/W}(\mu)\cdot (f+W)$ is in $V/W$, and $\pi_{V/W}(\mu)$ is a linear operator on $V/W$, further $\pi_{V/W}:\mu\mapsto \mu_V$ is a representation of $\mathcal M_c(G)$ on $V/W$. 
\vskip.2cm

It is obvious that $\pi_{V/W}$ is  a subrepresentation of $\pi_V$. Indeed, if $\mu$ is in $\Ker \pi_V$, then $\pi_V(\mu)=0$, that is, $\mu*f=\pi_V(\mu)*f=0$ for each $f$ in $V$. Then we have $\pi_{V/W}(\mu)(f+W)=\mu*f+W$, which is in $W$, hence $\mu$ is in $\Ker\pi_{V/W}$. 
\vskip.2cm
Our main result follows.

\begin{Thm}\label{main1}
	Let $G$ be a locally compact Abelian group which has no discrete subgroup of infinite torsion free rank. Spectral synthesis holds for a variety in $\mathcal C(G)$ if and only if spectral analysis holds for it. 
\end{Thm}

\begin{proof}
	If spectral synthesis holds for a variety, then obviously spectral analysis holds for it.
	\vskip.2cm
	
	For the converse, let $V$ be a variety and suppose that spectral analysis holds for $V$. Let $V_e$ denote the closed subspace generated by all exponential monomials in $V$: then $V_e$ is a variety. Let $\pi_{V/V_e}$ denote the natural representation of $\mathcal M_c(G)$ on  the space $V/V_e$. Then $\pi_{V/V_e}$ is a subrepresentation of the canonical representation $\pi_V$ of $\mathcal M_c(G)$ on $V$. By spectral analysis on $V$, we have that every maximal ideal in $\mathcal M_c(G)$ which includes $\Ker \pi_{V/V_e}$, is closed. If $V/V_e$ is nonzero, then, by spectral analysis, there exists at least one maximal ideal $M_0$ including $\Ker \pi_{V/V_e}$. The annihilator of $M_0$ in $V/V_e$ is a subvariety of $V/V_e$. It is easy to see, that every subvariety of $V/V_e$ is of the form $W/V_e$ with some subvariety $W$ of $V$ which includes $V_e$. So we have $\Ann M_0=W/V_e$ in $V/V_e$. The closed maximal ideal $M_0$ is generated by the measures $\delta_{-y}-m_0(y)\delta_0$ with $y$ in $G$, where $m_0$ is an exponential on $G$, which is in $W$ (see e.g. \cite[Section 2]{MR3502634}). It follows that, if $f+V_e$ is in $W/V_e$, then we have that $(\delta_{-y}-m_0(y)\delta_0)\cdot (f+V_e)$ is zero in $W/V_e$, i.e. the function $(\delta_{-y}-m_0(y)) *f$ is in $V_e$, for each $y$ in $G$. As spectral synthesis obviously holds for $V_e$, we have that the annihilator of $V_e$ satisfies
	$$
	\Ann V_e=\bigcap_{\Ann V_e\subseteq M_m}\bigcap_{k=0}^{\infty} M_m^{k+1},
	$$
	where the first intersection extends to all exponentials $m$ in $V$, by Theorem \ref{tor}.
	Each function in $M_0*f$ is annihilated by $\Ann V_e$, hence $f$ is annihilated by $\Ann V_e*M_0$. This means that $f$ is in $\Ann (\Ann V_e*M_0)$, that is $f$ is in
	$$
	\sum_m\,\sum_{k=0}^{\infty} \Ann (M_m^{k+1}\cdot  M_0)= \sum_{m\ne m_0}\,\sum_{k=0}^{\infty} \Ann (M_m^{k+1}\cap M_0)+\sum_{k=1}^{\infty}\Ann M_0^{k+1}=
	$$
	$$
	\sum_{m\ne m_0}\,\sum_{k=0}^{\infty} \Ann M_m^{k+1}+\Ann M_0+\sum_{k=1}^{\infty}\Ann M_0^{k+1}=
	$$
	$$
	\sum_{m\ne m_0}\,\sum_{k=0}^{\infty} \Ann M_m^{k+1}+\sum_{k=1}^{\infty}\Ann M_0^{k}=
	\sum_{m}\,\sum_{k=0}^{\infty} \Ann M_m^{k+1}.
	$$
The elements of the variety $\Ann M_m^{k+1}$ are $m$-exponential monomials. Consequently, the elements of the variety $\Ann (\Ann V_e*M_0)$ are limits of exponential polynomials in $V$.  We conclude that $f$ is a limit of exponential polynomials, hence it is in $V_e$, a contradiction. Our theorem is proved.
\end{proof}

	\end{document}